\documentclass[11pt]{article}
\usepackage{amssymb,float,amsthm,hyperref,xcolor,amsmath,tipa}
\hypersetup{
    colorlinks,
    citecolor=violet,
    filecolor=blue,
    linkcolor=blue,
    urlcolor=violet
}

\newtheorem{theorem}{Theorem}[section]
\newtheorem{lemma}[theorem]{Lemma}
\newtheorem{corollary}[theorem]{Corollary}
\newtheorem{proposition}[theorem]{Proposition}

\theoremstyle{definition}
\newtheorem{remark}[theorem]{Remark}

\newtheorem{problem}[theorem]{Open Question}

\usepackage[margin=1in]{geometry}
\counterwithin{figure}{section}

\title{Restricted generalized Schur numbers}
\author{Collier Gaiser\thanks{Department of Mathematics, Community College of Aurora, Aurora, CO 80011,  United States of America. Email: {\tt colliergaiser@gmail.com}}
}
\date{}
\begin{document}
\maketitle

\maketitle
\begin{abstract}
For $k\geq2$, let $S_r(k;\ell)$ be the smallest $n$, if exists, such that every $r$-coloring of $\{1,2,\ldots,n\}$ has a monochromatic solution $\mathcal{S}$ to the equation
\[
x_1+x_2+\cdots+x_k=x_{k+1}
\]
such that the number of distinct integers in $\mathcal{S}$ is exactly $\ell+1$. We prove that, if $\ell\geq2$ is fixed, then
\[
S_2(k;\ell)= k^2+\left[\frac{(\ell+1)(\ell-2)}{2}+2\right]k+\ell(\ell-2)
\]
for all large enough $k$. In particular, we have $S_2(k;2)=k^2+2k$ for all $k\geq3$.
\end{abstract}


\maketitle

\section{Introduction}\label{Section:Intro}
For positive integers $n$ and $r$, an $r$-coloring of $\{1,2,\ldots,n\}$ is a function \[\Delta:\{1,2,\ldots,n\}\to\{1,2,\ldots,r\}.\] Given an $r$-coloring of $\{1,2,\ldots,n\}$, a solution $\mathcal{S}$ in $\{1,2,\ldots,n\}$ to an equation is monochromatic if all integers in $\mathcal{S}$ are mapped to the same integer; that is, they have the same color. Throughout this paper, we use $\mathbb{N}:=\{1,2,\ldots\}$ to denote the set of all positive integers.

In arithmetic Ramsey theory, Schur's theorem \cite{Schur1916} states that, for any $r\in\mathbb{N}$, there exists a smallest $n\in\mathbb{N}$ such that every $r$-coloring of $\{1,2,\ldots,n\}$ has a monochromatic solution to the three-variable equation $x_1+x_2=x_3$. Rado \cite{Rado1933} generalized Schur's theorem and proved a general criterion for any system of linear homogeneous equations to have this property. In particular, for $k,r\in\mathbb{N}$ with $k\geq2$ (the case $k=1$ is trivial and hence we only consider $k\geq2$ in this paper), there exists a smallest $n\in\mathbb{N}$ such that every $r$-coloring of $\{1,2,\ldots,n\}$ has a monochromatic solution to the $(k+1)$-variable equation
\begin{equation}\label{Equaiton:Main}
x_1+x_2+\cdots+x_k=x_{k+1}.
\end{equation}
Interestingly, there is a nice proof of the above result for (\ref{Equaiton:Main}) using a graph-theoretic version of Ramsey's theorem (see \cite[p.~223 and p.~236]{LR2014} or \cite[p.~159]{Robertson2000}).

For $k,r\in\mathbb{N}$ with $k\geq2$, let $S_r(k)$ be the smallest $n\in\mathbb{N}$ such that every $r$-coloring of $\{1,2,\ldots,n\}$ has a monochromatic solution to (\ref{Equaiton:Main}), where $x_1,x_2,\ldots,x_k$ are not necessarily distinct. Due to Schur's theorem, the numbers $S_r(2)$ are called Schur numbers and, in general, the numbers $S_r(k)$ are called generalized Schur numbers \cite[p.~235]{LR2014}. Zn\'{a}m \cite{Znam1966} proved that, for all $k,r\in\mathbb{N}$ with $k\geq2$, we have
\begin{equation}\label{Equation:Znam_Bound}
S_r(k)\geq\frac{k-1}{k}[(k+1)^r-1]+1.
\end{equation}
Matching the above lower bound, Zn\'{a}m \cite{Znam1967} (see also \cite{BB1982}) also proved that $S_2(k)=k^2+k-1$ for all $k\geq2$, and Boza, Mar\'{i}n, Revuelta, and Sanz \cite{BMRS2019} confirmed that $S_3(k)=k^3+2k^2-2$ for all $k\geq2$.

Since the variables $x_1,x_2,\ldots,x_k$ are not necessarily distinct, the values of $S_r(k)$ depend crucially on the ``almost" trivial solutions such as $x_1=1$, $x_2=1$, $\ldots$ , $x_k=1$, $x_{k+1}=k$. In the opposite direction, for $k,r\in\mathbb{N}$ with $k\geq2$, let $S^*_r(k)$ be the smallest $n\in\mathbb{N}$ such that every $r$-coloring of $\{1,2,\ldots,n\}$ has a monochromatic solution to (\ref{Equaiton:Main}), where $x_1,x_2,\ldots,x_k$ are required to be distinct. Irving \cite{Irving1973} showed that $S_r^*(k)$ exists for all $k,r\in\mathbb{N}$ with $k\geq2$, and provided an upper bound for $S_r^*(k)$. Since the variables are required to be distinct, $S_r^*(k)$ is more difficult to determine; currently, we don't have a lower bound like (\ref{Equation:Znam_Bound}) for $S_r^*(k)$. The best lower bounds for $2$- and $3$-colorings are proven recently by Ahmed, Boza, Revuelta, and Sanz \cite{ABRS2023}: $S_2^*(k)\geq(k^3+4k^2-5k+2)/2$ for all $k\geq5$, and $S_3^*(k)\geq(k^4+5k^3-8k+4)/2$ for all $k\geq8$.

In this paper, we introduce and study generalized Schur numbers when the number of distinct integers in the solution is fixed and independent of $k$. For any solution $\mathcal{S}$ to (\ref{Equaiton:Main}), let $c(\mathcal{S})$ denote the number of distinct integers in $\mathcal{S}$. With this notation, for a solution $\mathcal{S}$ to (\ref{Equaiton:Main}) in $\mathbb{N}$, there are exactly $c(\mathcal{S})-1$ positive integers assigned to the variables $x_1,x_2,\ldots,x_k$ on the left hand side of (\ref{Equaiton:Main}) when $k\geq2$. For $r,k,\ell\in\mathbb{N}$ with $\ell\leq k$ and $k\geq2$, let $S_r(k,\ell)$ be the smallest $n\in\mathbb{N}$, if exists, such that every $r$-coloring of $\{1,2,\ldots,n\}$ has a monochromatic solution $\mathcal{S}$ to (\ref{Equaiton:Main}) such that $c(\mathcal{S})=\ell+1$. We call the numbers $S_r(k;\ell)$ \textit{restricted generalized Schur numbers}. Notice that we have $S_r(k;k)=S_r^*(k)$ for all $k\geq2$.

The existence of $S_r(k;\ell)$ is not immediate from existing results in arithmetic Ramsey theory. Hence, for completeness, we first show that, given $k,\ell\in\mathbb{N}$ with $\ell\leq k$ and $k\geq2$, the restricted generalized Schur number $S_r(k;\ell)$ exists if and only if $\ell\geq2$ or $r=1$. Then we prove the following main result on $2$-color restricted generalized Schur numbers:

\begin{theorem}\label{Theorem:Main}
For all $k,\ell\in\mathbb{N}$ with $2\leq\ell\leq k$, we have
    \[
    S_2(k;\ell)\geq k^2+\left[\frac{(\ell+1)(\ell-2)}{2}+2\right]k+\ell(\ell-2);
    \]
furthermore, fixing $\ell$, if $k$ is large enough, then we have
 \[
    S_2(k;\ell)= k^2+\left[\frac{(\ell+1)(\ell-2)}{2}+2\right]k+\ell(\ell-2).
 \]
In particular, we have $S_2(k;2)=k^2+2k$ for all $k\geq3$.
\end{theorem}
We note that even though Theorem~\ref{Theorem:Main} is not true for $\ell=1$, the formula nevertheless matches the generalized Schur number $S_2(k)=k^2+k-1$ if we set $\ell=1$. In Section~\ref{Section:Conluding}, we note that when the definition of restricted generalized Schur numbers is relaxed to include all solutions $\mathcal{S}$ to (\ref{Equaiton:Main}) such that $c(\mathcal{S})\geq\ell+1$, a corollary of Theorem~\ref{Theorem:Main} contains the generalized Schur number $S_2(k)=k^2+k-1$ as a special case.

The rest of this paper is organized as follows. We first establish a result on the existence of $S_r(k;\ell)$ in Section~\ref{Section:Existence}. In Section~\ref{Section:Lower_Bound}, we prove the lower bound in Theorem~\ref{Theorem:Main}. We prove the upper bound in Theorem~\ref{Theorem:Main} for $\ell=2$ in Section~\ref{Section:Upper_Bound_2} and hence establish that $S_2(k;2)=k^2+2k$ for all $k\geq3$. The upper bound in Theorem~\ref{Theorem:Main} for $\ell\geq3$ is proved in Section~\ref{Section:Upper_Bound_>=3}. In Section~\ref{Section:Conluding}, we provide a lower bound for $S_3(k;2)$, discuss a relaxed version of restricted generalized Schur numbers, and pose two open questions.
\section{The Existence of $S_r(k;\ell)$}\label{Section:Existence} 
We start with a basic observation.
\begin{lemma}\label{Lemma:Disjunctive}
Let $A\subseteq\mathbb{N}$ be a finite set of positive integers, and let $k,\ell\in\mathbb{N}$ with $2\leq\ell\leq k$. Then the following are equivalent:
\begin{itemize}
\item[(a)] $A$ has a solution $\mathcal{S}$ to (\ref{Equaiton:Main}) such that $c(\mathcal{S})=\ell+1$;
\item[(b)] $A$ has a solution to a system
\[
\left\{
\begin{split}
&a_1x_1+a_2x_2+\cdots+a_\ell x_\ell=x_{\ell+1},\\
&x_1<x_2<\cdots<x_\ell
\end{split}\right.
\]
for some $a_1,a_2,\ldots,a_\ell\in\mathbb{N}$ such that $a_1+a_2+\cdots+a_\ell=k$.
\end{itemize}
\end{lemma}
\begin{proof}
We first suppose (a) is true. Then there exists a subset $\{u_1,u_2,\ldots,u_{\ell+1}\}\subseteq A$ and $a_1,a_2,\ldots,a_\ell\in\mathbb{N}$ such that $u_1<u_2<\ldots<u_{\ell+1}$, $a_1+a_2+\cdots+a_\ell=k$, and
\begin{equation}\label{Equation:ForLemma}
\underbrace{u_1+u_1+\cdots+u_1}_{a_1\text{ times}}+\underbrace{u_2+u_2+\cdots+u_2}_{a_2\text{ times}}+\cdots+\underbrace{u_\ell+u_\ell+\cdots+u_\ell}_{a_\ell\text{ times}}=u_{\ell+1}.
\end{equation}
It follows that $a_1u_1+a_2u_2+\cdots+a_\ell u_\ell=u_{\ell+1}$ and hence (b) holds.

Now suppose (b) is true. Then there exists a subset $\{u_1,u_2,\ldots,u_{\ell+1}\}\subseteq A$ such that $u_1<u_2<\ldots<u_{\ell+1}$ and $a_1u_1+a_2u_2+\cdots+a_\ell u_\ell=u_{\ell+1}$ for some $a_1,a_2,\ldots,a_\ell\in\mathbb{N}$ with $a_1+a_2+\cdots+a_\ell=k$. Since $a_1u_1+a_2u_2+\cdots+a_\ell u_\ell=u_{\ell+1}$ is equivalent to (\ref{Equation:ForLemma}), we see that (a) holds.
\end{proof}

\begin{remark}
Johnson and Schaal \cite{JS2005} introduced the concept of $r$-color \textit{disjunctive Rado numbers}, which is defined as the smallest $n\in\mathbb{N}$, if exists, such that every $r$-coloring of $\{1,2,\ldots,n\}$ has a monochromatic solution to at least one equation in a set of equations. By Lemma~\ref{Lemma:Disjunctive}, $S_r(k;\ell)$, if exists, is the smallest $n\in\mathbb{N}$ such that every $r$-coloring of $\{1,2,\ldots,n\}$ has a monochromatic solution to at least one system in the following set:
\begin{equation}\label{Equation:Disjunctive}
\left\{
\begin{array}{ll}
a_1x_1+a_2x_2+\cdots+a_\ell x_\ell=x_{\ell+1},\\x_1<x_2<\cdots<x_\ell
\end{array}:a_1,a_2,\ldots,a_\ell\in\mathbb{N},a_1+a_2+\cdots+a_\ell=k\right\}.
\end{equation}
Hence, the restricted generalized Schur numbers $S_r(k;\ell)$ are disjunctive Rado numbers for (\ref{Equation:Disjunctive}). For more results on disjunctive Rado numbers, see \cite{DMT2022,DMT2024,DT2025,KS2003}.
\end{remark}

The following result is a direct consequence of a result of Hindman and Leader \cite[Theorem 2]{HL1998} (see also \cite[Theorem 2.2]{GGL2014}).

\begin{lemma}[\cite{HL1998}]\label{Lemma:StronglyRegular}
Let $k,a_1,a_2,\ldots,a_k\in\mathbb{N}$ with $k\geq2$. If every finite coloring of $\mathbb{N}$ has a monochromatic solution to the equation $a_1x_1+a_2x_2+\cdots+a_kx_k=x_{k+1}$, where $x_1,x_2,\ldots,x_k$ are not necessarily distinct, then every finite coloring of $\mathbb{N}$ has a monochromatic solution to the system $a_1x_1+a_2x_2+\cdots+a_kx_k=x_{k+1}$ and $x_1<x_2<\cdots<x_k$.
\end{lemma}

Now we prove a result on the existence of restricted generalized Schur numbers $S_r(k;\ell)$.

\begin{theorem}
Let $k,\ell,r\in\mathbb{N}$ with $\ell\leq k$ and $k\geq2$. Then $S_r(k;\ell)$ exists if and only if $\ell\geq2$ or $r=1$.
\end{theorem}
\begin{proof}
It is easy to see that $S_r(k;\ell)$ exists when $r=1$. So we assume that $r\geq2$. 

We first show that $S_r(k;1)$ does not exist for any $r\geq2$. We do so by constructing a $2$-coloring of $\mathbb{N}$ that does not have a monochromatic solution $\mathcal{S}$ to (\ref{Equaiton:Main}) with $c(\mathcal{S})=2$. For all $i\in\mathbb{N}$, let
\[
A_i=\{k^{i-1},k^{i-1}+1,\ldots,k^i-1\}.
\]
Let $\Delta:\mathbb{N}\to\{R,B\}$ be a $2$-coloring such that $\Delta(a)=R$ for all $a\in A_i$ with $i$ even and $\Delta(a)=B$ for all $a\in A_i$ with $i$ odd. Let $a\in A_i$ for some $i\in\mathbb{N}$. Then we have $k^{i-1}\leq a<k^i$ and hence
\[
k^i\leq \underbrace{a+a+\cdots+a}_{k\text{ times}}=ka<k^{i+1}.
\]
It follows that $ka\in A_{i+1}$. Since the integers in $A_i$ and $A_{i+1}$ have different colors, $\Delta$ does not have a monochromatic solution $\mathcal{S}$ to (\ref{Equaiton:Main}) with $c(\mathcal{S})=2$. This proves that $S_2(k;1)$ does not exist. Since a $2$-coloring is also an $r$-coloring for any $r\geq2$, we see that $S_r(k;1)$ does not exist for any $r\geq2$.

Now we show that $S_r(k;\ell)$ exists for all $k,\ell,r\in\mathbb{N}$ with $2\leq\ell\leq k$ and $r\geq2$. Let $a_1=1$ and $a_2,a_3,\ldots,a_\ell\in\mathbb{N}$ such that $1+a_2+a_3+\cdots+a_{\ell}=k$. By Rado's single equation theorem (see \cite[Theorem 9.5]{LR2014}), every $r$-coloring of $\mathbb{N}$ has a monochromatic solution to
\begin{equation}\label{Equation:Secondary}
x_1+a_2x_2+a_3x_3+\cdots+a_{\ell}x_\ell=x_{\ell+1},
\end{equation}
where $x_1,x_2,\ldots,x_\ell$ are not necessarily distinct. By Lemma~\ref{Lemma:StronglyRegular}, every $r$-coloring of $\mathbb{N}$ has a monochromatic solution to (\ref{Equation:Secondary}) with $x_1<x_2<\cdots<x_{\ell}$. Since $1+a_2+a_3+\cdots+a_{\ell}=k$, by Lemma~\ref{Lemma:Disjunctive}, every $r$-coloring of $\mathbb{N}$ has a monochromatic solution $\mathcal{S}$ to (\ref{Equaiton:Main}) such that $c(\mathcal{S})=\ell+1$. By the compactness principle (see \cite[Theorem 2.4]{LR2014}), we see that $S_r(k;\ell)$ exists.
\end{proof}
\section{Proof of Theorem~\ref{Theorem:Main}: Lower Bound}\label{Section:Lower_Bound}
Let $k,\ell\in\mathbb{N}$ with $2\leq\ell\leq k$. Write
\[
T:=1+2+\cdots+\ell+\underbrace{1+1+\cdots+1}_{k-\ell\text{ times}}-1=k+\frac{(\ell+1)(\ell-2)}{2},
\]
and 
\[
T':=k^2+\left[\frac{(\ell+1)(\ell-2)}{2}+2\right]k+\ell(\ell-2).
\]
Notice that we have
\[
T'=[(k+1)T+1]+\underbrace{1+1+\cdots+1}_{k-\ell\text{ times}}+1+2+\cdots+(\ell-1)=k(T+2)+\ell(\ell-2).
\]
We need to show that $S_r(k;\ell)\geq T'$. Let 
\[
A_R'=\left\{1,2,\ldots,T\right\},
\]
\[
A_R''=\left\{(k+1)T+1,(k+1)T+2,\ldots,T'-1
\right\},
\]
and
\[
A_B=\{T+1,T+2,\ldots,(k+1)T\}.
\]

Let $\Delta:\left\{1,2,\ldots,T'-1\right\}\to\{R,B\}$ be a $2$-coloring such that $\Delta(a)=R$ for all $a\in A_R'\cup A_R''$ and $\Delta(a)=B$ for all $a\in A_B$. We will show that $\Delta$ does not have a monochromatic solution $\mathcal{S}$ to (\ref{Equaiton:Main}) such that $c(S)=\ell+1$. It suffices to show that neither $A_R'\cup A_R''$ nor $A_B$ has a solution $\mathcal{S}$ to (\ref{Equaiton:Main}) such that $c(\mathcal{S})=\ell+1$. 

Let $a_1,a_2,\ldots,a_{\ell}\in A_R$ such that $a_1<a_2<\cdots<a_{\ell}$, and $t_1,t_2,\ldots,t_{\ell}\in\mathbb{N}$ such that $t_1+t_2+\cdots+t_\ell=k$. If $a_\ell\in A_R'$, then we have
\[
\begin{split}
\underbrace{a_1+a_1+\cdots+a_1}_{t_1\text{ times}}+\underbrace{a_2+a_2+\cdots+a_2}_{t_2\text{ times}}+\cdots+\underbrace{a_\ell+a_\ell+\cdots+a_\ell}_{t_\ell\text{ times}}<ka_\ell\leq kT<(k+1)T
\end{split}
\]
and
\[
\begin{split}
&\underbrace{a_1+a_1+\cdots+a_1}_{t_1\text{ times}}+\underbrace{a_2+a_2+\cdots+a_2}_{t_2\text{ times}}+\cdots+\underbrace{a_\ell+a_\ell+\cdots+a_\ell}_{t_\ell\text{ times}}\\\geq&\underbrace{1+1+\cdots+1}_{k-\ell\text{ times}}+1+2+\cdots+\ell=T+1,
\end{split}
\]
and hence $a_1+a_2+\cdots+a_{\ell}\in A_B$. If $a_\ell\in A_R''$, then we have
\[
\begin{split}
&\underbrace{a_1+a_1+\cdots+a_1}_{t_1\text{ times}}+\underbrace{a_2+a_2+\cdots+a_2}_{t_2\text{ times}}+\cdots+\underbrace{a_\ell+a_\ell+\cdots+a_\ell}_{t_\ell\text{ times}}\\\geq&\underbrace{1+1+\cdots+1}_{k-\ell\text{ times}}+1+2+\cdots+(\ell-1)+[(k+1)T+1]=T'\notin A_R'\cup A_R''.
\end{split}
\]
Hence $A_R'\cup A_R''$ does not have a solution $\mathcal{S}$ to (\ref{Equaiton:Main}) such that $c(\mathcal{S})=\ell+1$. 

Let $b_1,b_2,\ldots,b_{\ell}\in A_B$ such that $b_1<b_2<\cdots<b_{\ell}$, and $t_1,t_2,\ldots,t_{\ell}\in\mathbb{N}$ such that $t_1+t_2+\cdots+t_\ell=k$. Then we have
\[
\begin{split}
&\underbrace{b_1+b_1+\cdots+b_1}_{t_1\text{ times}}+\underbrace{b_2+b_2+\cdots+b_2}_{t_2\text{ times}}+\cdots+\underbrace{b_\ell+b_\ell+\cdots+b_\ell}_{t_\ell\text{ times}}\\\geq&\underbrace{(T+1)+(T+1)+\cdots+(T+1)}_{k-\ell\text{ times}}+(T+1)+(T+2)+\cdots+(T+\ell)=(k+1)T+1\notin A_B.
\end{split}
\]
Hence $A_B$ does not have a solution $\mathcal{S}$ to (\ref{Equaiton:Main}) such that $c(\mathcal{S})=\ell+1$. This completes the proof that $S_2(k;\ell)\geq T'$.
\section{Proof of Theorem~\ref{Theorem:Main}: Upper Bound When $\ell=2$}\label{Section:Upper_Bound_2}
In this section, we show that $S_2(k;2)\leq k^2+2k$ for all $k\geq3$. Combining this upper bound with the lower bound in Theorem~\ref{Theorem:Main}, we then have $S_2(k;2)=k^2+2k$ for all $k\geq3$. Since $S_2(2;2)=S_2(2)=9$ (see \cite[p.~349]{ABRS2023}), we have a complete answer to $S_2(k;2)$.

We note that there are two reasons to prove the upper bound in Theorem~\ref{Theorem:Main} for the case $\ell=2$ independently. One reason is to get a better threshold for $k$ such that $S_2(k;2)=k^2+2k$; and the other reason is that the proof strategy for the case when $\ell\geq3$ (more specifically, see Case 1 and Case 3 in Section~\ref{Section:Upper_Bound_>=3}) does not apply to the case when $\ell=2$.

Let $\Delta:\{1,2,\ldots,k^2+2k\}\to\{R,B\}$ be a $2$-coloring. Suppose, by way of contradiction, that $\Delta$ does not have a monochromatic solution $\mathcal{S}$ to (\ref{Equaiton:Main}) such that $c(\mathcal{S})=3$. Without loss of generality, we assume that $\Delta(1)=R$. There are four cases depending on the colors assigned to the integers $2$ and $3$.

\textbf{Case 1}: $\Delta(1)=\Delta(2)=\Delta(3)=R$. Since
\[
\underbrace{1+1+\cdots+1}_{k-1\text{ times}}+2=k+1,
\]
\[
\underbrace{1+1+\cdots+1}_{k-2\text{ times}}+2+2=k+2,
\]
and
\[
\underbrace{2+2+\cdots+2}_{k-1\text{ times}}+3=2k+1,
\]
we have $\Delta(k+1)=\Delta(k+2)=\Delta(2k+1)=B$. Since
\[
\underbrace{(k+1)+(k+1)+\cdots+(k+1)}_{k-1\text{ times}}+(k+2)=k^2+k+1,
\]
we have $\Delta(k^2+k+1)=R$. Since
\[
\underbrace{1+1+\cdots+1}_{k-1\text{ times}}+(k^2+k+1)=k^2+2k,
\]
we have $\Delta(k^2+2k)=B$. Since
\[
\underbrace{(k+1)+(k+1)+\cdots+(k+1)}_{k-1\text{ times}}+(2k+1)=k^2+2k,
\]
we have a monochromatic solution $\mathcal{S}$ to (\ref{Equaiton:Main}) such that $c(S)=3$, which is a contradiction.

\textbf{Case 2}: $\Delta(1)=\Delta(2)=R$ and $\Delta(3)=B$. Similar to Case 1, we have $\Delta(k+1)=\Delta(k+2)=\Delta(k^2+2k)=B$ and $\Delta(k^2+k+1)=R$. Since
\[
\underbrace{3+3+\cdots+3}_{k-1\text{ times}}+(k+2)=4k-1,
\]
we have $\Delta(4k-1)=R$. Since 
\[
\underbrace{2+2+\cdots+2}_{k-1\text{ times}}+(2k+1)=4k-1,
\]
we have $\Delta(2k+1)=B$. Now, similar to Case 1, we have a monochromatic solution $\mathcal{S}$ to (\ref{Equaiton:Main}) such that $c(\mathcal{S})=3$, which is a contradiction. We note that $k^2+2k\geq4k-1$ for all $k\geq3$.

\textbf{Case 3}: $\Delta(1)=R$ and $\Delta(2)=\Delta(3)=B$. Since 
\[
\underbrace{2+2+\cdots+2}_{k-1\text{ times}}+3=2k+1
\]
and
\[
2+\underbrace{3+3+\cdots+3}_{k-1\text{ times}}=3k-1,
\]
we have $\Delta(2k+1)=\Delta(3k-1)=R$. Since
\[
\underbrace{1+1+\cdots+1}_{k-1\text{ times}}+(2k+1)=3k,
\]
we have $\Delta(3k)=B$. Since
\[
\underbrace{2+2+\cdots+2}_{k-1\text{ times}}+3k=5k-2,
\]
we have $\Delta(5k-2)=R$. Since
\[
\underbrace{1+1+\cdots+1}_{k-2\text{ times}}+2k+2k=5k-2
\]
and
\[
\underbrace{1+1+\cdots+1}_{k-1\text{ times}}+(3k-1)=4k-2,
\]
we have $\Delta(2k)=\Delta(4k-2)=B$. Since
\[
\underbrace{2+2+\cdots+2}_{k-1\text{ times}}+2k=4k-2,
\]
we have a monochromatic solution $\mathcal{S}$ to (\ref{Equaiton:Main}) such that $c(\mathcal{S})=3$, which is a contradiction. We note that $k^2+2k\geq5k-2$ for all $k\geq3$.

\textbf{Case 4}: $\Delta(1)=\Delta(3)=R$ and $\Delta(2)=B$. Since
\[
\underbrace{1+1+\cdots+1}_{k-1\text{ times}}+3=k+2
\]
and
\[
1+\underbrace{3+3+\cdots+3}_{k-1\text{ times}}=3k-2,
\]
we have $\Delta(k+2)=\Delta(3k-2)=B$. Since
\[
\underbrace{2+2+\cdots+2}_{k-1\text{ times}}+(k+2)=3k,
\]
we have $\Delta(3k)=R$. Since
\[
\underbrace{1+1+\cdots+1}_{k-1\text{ times}}+3k=4k-1,
\]
we have $\Delta(4k-1)=B$. Since
\[
\underbrace{2+2+\cdots+2}_{k-1\text{ times}}+(4k-1)=6k-3,
\]
we have $\Delta(6k-3)=R$. Since
\[
\underbrace{3+3+\cdots+3}_{k-1\text{ times}}+3k=6k-3,
\]
we have a monochromatic solution $\mathcal{S}$ to (\ref{Equaiton:Main}) such that $c(S)=3$, which is a contradiction. We note that $k^2+2k\geq6k-3$ for all $k\geq3$.

Since $k\geq3$, all the integers used in the identities are in $\{1,2,\ldots,k^2+2k\}$; and it is routine to check that all the identities involve two distinct integers on the left hand side and hence all the identities in the proof are valid. This completes the proof that $S_2(k;2)\leq k^2+2k$ for all $k\geq3$.
\section{Proof of Theorem~\ref{Theorem:Main}: Upper Bound When $\ell\geq3$}\label{Section:Upper_Bound_>=3}
Let $\ell\geq3$. Similar to Section~\ref{Section:Lower_Bound}, we write
\[
T:=1+2+\cdots+\ell+\underbrace{1+1+\cdots+1}_{k-\ell\text{ times}}-1=k+\frac{(\ell+1)(\ell-2)}{2},
\]
and 
\[
T':=k^2+\left[\frac{(\ell+1)(\ell-2)}{2}+2\right]k+\ell(\ell-2).
\]
We will prove that $S_2(k,\ell)\leq T'$ whenever $k\geq K$ for some large enough $K$, which will be determined later. For the time being, to ease the proof, we at least require that $K\geq3\ell-1$. At the end of the proof, we will see that $K$ might need to be larger.

Let $\Delta:\{1,2,\ldots,T'\}\to\{R,B\}$ be a $2$-coloring. By the pigeonhole principle, there exist at least $\ell$ integers in $\{1,2,\ldots,2\ell-1\}$ with the same color. We select the ``smallest" $u_1,u_2,\ldots,u_{\ell}\in \{1,2,\ldots,2\ell-1\}$ such that $u_1<u_2<\cdots<u_\ell$ and $\Delta(u_1)=\Delta(u_2)=\cdots=\Delta(u_\ell)$ using the \textit{lexicographic order} \cite[p.~5]{EG2021}; that is, under the lexicographic order, we have $u_1u_2\ldots u_\ell\leq w_1w_2\ldots w_\ell$ for all $\ell$ integers $w_1,w_2,\ldots,w_\ell\in\{1,2,\ldots,2\ell-1\}$ with $w_1<w_2<\cdots<w_\ell$ and $\Delta(w_1)=\Delta(w_2)=\cdots=\Delta(w_\ell)$. By doing so, if there exists $v\in\{1,2,\ldots,2\ell-1\}$ such that $u_i<v<u_{i+1}$ for some $i\in\{1,2,\ldots,\ell-1\}$, then we have $\Delta(v)\neq\Delta(u_i)=\Delta(u_{i+1})$; and if $v<u_1$, then $\Delta(v)\neq\Delta(u_1)$. This is because otherwise we wouldn't select $u_1,u_2,\ldots,u_\ell$ under the lexicographic order. Without loss of generality, we assume that $\Delta(u_1)=\Delta(u_2)=\cdots=\Delta(u_\ell)=R$.

\textbf{Case 1}: $u_i=i$ for all $i\in\{1,2,\ldots,\ell\}$. So we have $\Delta(1)=\Delta(2)=\cdots=\Delta(\ell)=R$. Since we require $K\geq3\ell-1$, we have $k\geq3\ell-2$, $k>\ell$, and $k>3$ for all $k\geq K$. For all $j\in\{0,1,2,\ldots,2\ell-2\}$, since
\[
1+2+\cdots+\ell+\underbrace{1+1+\cdots+1}_{k-\ell-j\text{ times}}+\underbrace{2+2+\cdots+2}_{j\text{ times}}=T+j+1,
\]
we have $\Delta\left(T+j+1\right)=B$. Since
\[
(T+1)+(T+2)+\cdots+(T+\ell)+\underbrace{(T+1)+(T+1)+\cdots+(T+1)}_{k-\ell\text{ times}}=(k+1)T+1,
\]
we have $\Delta\left((k+1)T+1\right)=R$. When $\ell=3$, we have
\[
(T+1)+(T+4)+(T+4)+\underbrace{(T+2)+(T+2)+\cdots+(T+2)}_{k-3\text{ times}}=T';
\]
and when $\ell\geq4$, we have
\[
(T+1)+(T+4)+(T+4)+(T+7)+(T+9)\cdots+(T+2\ell-1)+\underbrace{(T+2)+(T+2)+\cdots+(T+2)}_{k-\ell\text{ times}}=T'.
\]
It follows that $\Delta(T')=R$. Since
\[
1+2+\cdots+(\ell-1)+[(k+1)T+1]+\underbrace{1+1+\cdots+1}_{k-\ell\text{ times}}=T',
\]
we have a monochromatic solution $\mathcal{S}$ to (\ref{Equaiton:Main}) with $c(\mathcal{S})=\ell+1$, which is a contradiction. 

\textbf{Case 2}: $u_i=2i-1$ for all $i\in\{1,2,\ldots,\ell\}$. So we have $\Delta(1)=\Delta(3)=\cdots=\Delta(2\ell-1)=R$ and $\Delta(2)=\Delta(4)=\cdots=\Delta(2\ell-2)=B$. Since we require $K\geq3\ell-1$, we have $k\geq2\ell-1$ for all $k\geq K$. Since
\[
1+3+\cdots+(2\ell-1)+\underbrace{1+1+\cdots+1}_{k-\ell\text{ times}}=k+\ell^2-\ell,
\]
we have $\Delta(k+\ell^2-\ell)=B$. Since
\[
2+4+\cdots+(2\ell-2)+\underbrace{2+2+\cdots+2}_{k-\ell\text{ times}}+(k+\ell^2-\ell)=3k+2\ell^2-4\ell,
\]
we have $\Delta(3k+2\ell^2-4\ell)=R$. Since
\[
1+3+\cdots+(2\ell-3)+\underbrace{1+1+\cdots+1}_{k-\ell\text{ times}}+(3k+2\ell^2-4\ell)=4k+3\ell^2-7\ell+1,
\]
we have $\Delta(4k+3\ell^2-7\ell+1)=B$. Since
\[
2+4+\cdots+(2\ell-2)+\underbrace{2+2+\cdots+2}_{k-\ell\text{ times}}+(4k+3\ell^2-7\ell+1)=6k+4\ell^2-10\ell+1,
\]
we have $\Delta(6k+4\ell^2-10\ell+1)=R$. Since
\[
3+5+\cdots+(2\ell-1)+3+5+\cdots+(2\ell-1)+\underbrace{3+3+\cdots+3}_{k-\ell-(\ell-1)\text{ times}}+(3k+2\ell^2-4\ell)=6k+4\ell^2-10\ell+1,
\]
we have a monochromatic solution $\mathcal{S}$ to (\ref{Equaiton:Main}) with $c(S)=\ell+1$, which is a contradiction. We note that, since $\ell\geq3$, the largest integer used in the identities for Case 2 is $6k+4\ell^2-10\ell+1$. 

\textbf{Case 3}: There exists some $i\in\{1,2,\ldots,\ell-1\}$ such that $u_{i+1}=u_{i}+1$, and there exists some $v\in\{1,2,\ldots,\ell\}$ such that $\Delta(v)=B$. Let $j$ be the smallest such that $u_{j+1}=u_{j}+1$ and $v$ be the smallest such that $\Delta(v)=B$. Notice that we have $u_{\ell}>v$. Write $U:=u_1+u_2+\cdots+u_\ell$. Since we require $K\geq3\ell-1$, we have $k>\ell$ and $k\geq2\ell-2$ for all $k\geq K$.

For all $t\in\{0,1,\ldots,\ell-2\}$, since
\[
u_1+u_2+\cdots+u_{\ell}+\underbrace{u_j+u_j+\cdots+u_j}_{k-\ell-t\text{ times}}+\underbrace{u_{j+1}+u_{j+1}+\cdots+u_{j+1}}_{t\text{ times}}=U+(k-\ell)u_j+t,
\]
we have $\Delta(U+(k-\ell)u_j+t)=B$. Write
\[
V:=(k-\ell+1)v+(\ell-1)U+(\ell-1)(k-\ell)u_j+\frac{(\ell-1)(\ell-2)}{2}.
\]
Since
\[
\underbrace{v+v+\cdots+v}_{k-(\ell-1)\text{ times}}+[U+(k-\ell)u_j]+[U+(k-\ell)u_j+1]+\cdots+[U+(k-\ell)u_j
+\ell-2]=V
\]
and
\[
\begin{split}
&\underbrace{v+v+\cdots+v}_{k-\ell\text{ times}}+[U+(k-\ell)u_j]\\&+[U+(k-\ell)u_j]+[U+(k-\ell)u_j+1]+\cdots+[U+(k-\ell)u_j+\ell-2]\\=&V+U+(k-\ell)u_j-v,
\end{split}
\]
we have $\Delta(V)=\Delta(V+U+(k-\ell)u_j-v)=R$.

\textit{Subcase 3.1}: $j\leq\ell-2$. Notice that $u_\ell-v\leq (2\ell-1)-1=2\ell-2$ and, since we require that $K\geq3\ell-1$, we have $k\geq\ell+u_\ell-v$ for all $k\geq K$. Since
\[
u_1+u_2+\cdots+u_{\ell-1}+\underbrace{u_j+u_j+\cdots+u_j}_{k-\ell-(u_{\ell}-v)\text{ times}}+\underbrace{u_{j+1}+u_{j+1}+\cdots+u_{j+1}}_{u_{\ell}-v\text{ times}}+V=V+U+(k-\ell)u_j-v,
\]
we have a monochromatic solution $\mathcal{S}$ to (\ref{Equaiton:Main}) with $c(\mathcal{S})=\ell+1$, which is a contradiction.

\textit{Subcase 3.2}: $\ell\geq4$ and $j=\ell-1\geq3$. In this case, we have $v=1$ or $2$ and hence $u_2>v$. Similar to Subcase 3.1, since we require $K\geq3\ell-1$, we have $k\geq\ell+u_2-v$ for all $k\geq K$. Since 
\[
u_1+u_3+u_4+\cdots+u_{\ell}+\underbrace{u_j+u_j+\cdots+u_j}_{k-\ell-(u_{2}-v)\text{ times}}+\underbrace{u_{j+1}+u_{j+1}+\cdots+u_{j+1}}_{u_{2}-v\text{ times}}+V=V+U+(k-\ell)u_j-v,
\]
we have a monochromatic solution $\mathcal{S}$ to (\ref{Equaiton:Main}) with $c(\mathcal{S})=\ell+1$, which is a contradiction.

\textit{Subcase 3.3}: $\ell=3$ and $j=2$. Then we have $(u_1,u_2,u_3)\in\{(2,4,5),(1,4,5),(1,3,4)\}$. We will verify that each scenario leads to a contradiction. Notice that since $\ell=3$ and we require $K\geq3\ell-1$, we have $k\geq8$ for all $k\geq K$.

\textsl{Subsubcase 3.3.1}: $(u_1,u_2,u_3)=(2,4,5)$. Then $v=1$ and hence we have
\[
u_2+u_3+\underbrace{u_2+u_2+\cdots+u_2}_{k-4\text{ times}}+u_3+V=V+U+(k-3)u_j-v.
\]
It follows that we have a monochromatic solution $\mathcal{S}$ to (\ref{Equaiton:Main}) with $c(\mathcal{S})=4$, which is a contradiction.

\textsl{Subsubcase 3.3.2}: $(u_1,u_2,u_3)=(1,4,5)$. Then we have $\Delta(1)=\Delta(4)=\Delta(5)=R$ and $\Delta(2)=\Delta(3)=B$. Since
\[
\underbrace{1+1+\cdots+1}_{k-2\text{ times}}+4+5=k+7,
\]
we have $\Delta(k+7)=B$. Since
\[
2+\underbrace{3+3+\cdots+3}_{k-2\text{ times}}+(k+7)=4k+3,
\]
we have $\Delta(4k+3)=R$. Since 
\[
1+\underbrace{4+4+\cdots+4}_{k-7\text{ times}}+5+5+5+5+5+5=4k+3,
\]
we have a monochromatic solution $\mathcal{S}$ to (\ref{Equaiton:Main}) with $c(\mathcal{S})=4$, which is a contradiction.

\textsl{Subsubcase 3.3.3}: $(u_1,u_2,u_3)=(1,3,4)$. Then we have $\Delta(1)=\Delta(3)=\Delta(4)=R$ and $\Delta(2)=B$. Since
\[
\underbrace{1+1+\cdots+1}_{k-2\text{ times}}+3+4=k+5,
\]
\[
\underbrace{1+1+\cdots+1}_{k-3\text{ times}}+3+3+4=k+7,
\]
\[
\underbrace{1+1+\cdots+1}_{k-4\text{ times}}+3+3+3+4=k+9,
\]
and
\[
\underbrace{1+1+\cdots+1}_{k-4\text{ times}}+3+3+4+4=k+10,
\]
we have $\Delta(k+5)=\Delta(k+7)=\Delta(k+9)=\Delta(k+10)=B$. Since
\[
\underbrace{2+2+\cdots+2}_{k-2\text{ times}}+(k+5)+(k+7)=4k+8
\]
and
\[
\underbrace{2+2+\cdots+2}_{k-2\text{ times}}+(k+9)+(k+10)=4k+15,
\]
we have $\Delta(4k+8)=\Delta(4k+15)=R$. Since
\[
\underbrace{1+1+\cdots+1}_{k-2\text{ times}}+3+(4k+8)=5k+9
\]
and
\[
3+\underbrace{4+4+\cdots+4}_{k-2\text{ times}}+(4k+15)=8k+10,
\]
we have $\Delta(5k+9)=\Delta(8k+10)=B$. Since
\[
\underbrace{2+2+\cdots+2}_{k-2\text{ times}}+(k+5)+(5k+9)=8k+10,
\]
we have a monochromatic solution $\mathcal{S}$ to (\ref{Equaiton:Main}) with $c(\mathcal{S})=4$, which is a contradiction. We note that the largest integer used in the identities for Case 3 is $\max\{V+U+(k-\ell)u_j-v,8k+10\}$.

We still need to show that there exists a large enough $K$ such that all the identities we used are valid for all $k\geq K$. For Case 3, since
\[
U=u_1+u_2+\cdots+u_\ell\leq\ell+(\ell+1)+\cdots+(2\ell-1)=\ell(2\ell-1)-\frac{\ell(\ell-1)}{2},
\]
we have
\[
\begin{split}
&V+U+(k-\ell)u_j-v\\=&(k-\ell+1)v+(\ell-1)U+(\ell-1)(k-\ell)u_j+\frac{(\ell-1)(\ell-2)}{2}+U+(k-\ell)u_j-v\\\leq&(k-\ell)\ell+\ell\left[\ell(2\ell-1)-\frac{\ell(\ell-1)}{2}\right]+\ell(k-\ell)(2\ell-2)+\frac{(\ell-1)(\ell-2)}{2},
\end{split}
\]
which is linear in $k$ for a fixed $\ell$. Since $T'$ is quadratic in $k$, there exists $K\geq3\ell-1$ such that
\[
T'\geq\max\{6k+4\ell^2-10\ell+1,V+U+(k-\ell)u_j-v,8k+10\}
\]
for all $k\geq K$. In other words, there exists $K$ such that, for all $k\geq K$, the integers used in the identities are all in $\{1,2,\ldots,T'\}$. Hence, for all $k\in K$, all the identities used in this section involve only the integers in $\{1,2,\ldots,T'\}$. At the same time, since $K\geq3\ell-1$, all the identities in this section involve $\ell$ distinct integers on the left hand side. Hence we have $S_2(k;\ell)\leq T'$ for all $k\geq T$. This completes the proof of the upper bound of Theorem~\ref{Theorem:Main} when $\ell\geq3$.

\section{Discussions}\label{Section:Conluding}
The $2$-coloring in the proof of the lower bound in Theorem~\ref{Theorem:Main} can be augmented to provide lower bounds for $S_r(k;\ell)$ with $r\geq3$. As an example, we prove the following lower bound for $S_3(k;2)$:

\begin{proposition}
We have $S_3(k;2)\geq k^3+3k^2+k-1$ for all $k\geq2$.
\end{proposition}
\begin{proof}
Let $k\geq2$ and let
\[
A_R=\{1,2,\ldots,k\}\cup\{k^2+k+1,k^2+k+2,\ldots,k^2+2k-1\}\cup\{k^3+2k^2+1,k^3+2k^2+2,\ldots,k^3+2k^2+k-1\},
\]
\[
A_B=\{k+1,k+2,\ldots,k^2+k\}\cup\{k^3+2k^2+k,k^3+2k^2+k+1,\ldots,k^3+3k^2+k-2\},
\]
and
\[
A_G=\{k^2+2k,k^2+2k+1,\ldots,k^3+2k^2\}.
\]
Let $\Delta:\{1,2,\ldots,k^3+3k^2+k-2\}\to\{R,B,G\}$ be a $3$-coloring such that $\Delta(a)=R$ for all $a\in A_R$, $\Delta(a)=B$ for all $a\in A_B$, and $\Delta(a)=G$ for all $a\in A_G$.

Similar to the proof of the lower bound in Theorem~\ref{Theorem:Main}, it is straightforward to check that none of $A_R$, $A_B$, or $A_G$ has a solution $\mathcal{S}$ to (\ref{Equaiton:Main}) such that $c(\mathcal{S})=3$. Hence $\Delta$ does not have a solution $\mathcal{S}$ to (\ref{Equaiton:Main}) such that $c(\mathcal{S})=3$. This completes the proof.
\end{proof}

An open question is whether $S_3(k;2)= k^3+3k^2+k-1$ for all large enough $k$.
\begin{problem}
Is it true that $S_3(k;2)= k^3+3k^2+k-1$ for all large enough $k$?
\end{problem}

For any $k,\ell,r\in\mathbb{N}$ with $\ell\leq k$ and $k\geq2$, let $S_r(k;\tilde{\ell})$ be the smallest $n\in\mathbb{N}$ such that every $r$-coloring of $\{1,2,\ldots,n\}$ has a monochromatic solution $\mathcal{S}$ to (\ref{Equaiton:Main}) such that $c(\mathcal{S})\geq\ell$. Notice that $S_r(k;\tilde{\ell})$ always exists because we have
\[
S_r(k;\tilde{\ell})\leq S_r(k;\tilde{k})=S_r(k;k)=S_r^*(k)
\]
and, as mentioned in Section~\ref{Section:Intro}, $S_r^*(k)$ exists. 

By the definitions of $S_r(k;\ell)$ and $S_r(k;\tilde{\ell})$, it is clear that $S_r(k;\tilde{\ell})\leq S_r(k;\ell)$ and $S_r(k;\tilde{1})=S_r(k)$. It is also straightforward to see that the lower bound in Theorem~\ref{Theorem:Main} also holds for $S_r(k;\tilde{\ell})$. Hence, we have the following:

\begin{corollary}\label{Theorem:Corollary}
For all $k,\ell\in\mathbb{N}$ with $\ell\leq k$ and $k\geq2$, we have
    \[
    S_2(k;\tilde{\ell})\geq k^2+\left[\frac{(\ell+1)(\ell-2)}{2}+2\right]k+\ell(\ell-2);
    \]
furthermore, fixing $\ell$, if $k$ is large enough, then we have
 \[
    S_2(k;\tilde{\ell})= k^2+\left[\frac{(\ell+1)(\ell-2)}{2}+2\right]k+\ell(\ell-2).
    \]
In particular, we have $S_2(k;\tilde{2})=k^2+2k$ for all $k\geq3$.
\end{corollary}
Corollary~\ref{Theorem:Corollary} generalizes, for large enough $k$, the fact that $S_2(k;\tilde{1})=S_2(k)=k^2+k-1$. Similar to $S_2(k;2)$, we have $S_2(2;\tilde{2})=9$ and $S_2(k;\tilde{2})=k^2+2k$ for all $k\geq3$. When $\ell\geq3$, the exact thresholds for $k$ in Theorem~\ref{Theorem:Main} and Corollary~\ref{Theorem:Corollary} remain open.

\begin{problem}
Let $\ell\geq3$. What is the smallest $K\in\mathbb{N}$ such that
\[
    S_2(k;\ell)= k^2+\left[\frac{(\ell+1)(\ell-2)}{2}+2\right]k+\ell(\ell-2)
    \]
for all $k\geq K$? What about $S_2(k;\tilde{\ell})$?
\end{problem}


\begin{thebibliography}{10}
\bibitem{ABRS2023} T. Ahmed, L. Boza, M. P. Revuelta, and M. I. Sanz, Exact values and lower bounds on the $n$-color weak Schur numbers for $n=2,3$, \textit{Ramanujan J.} \textbf{62} (2023), 347-363.

\bibitem{BB1982} A. Beutelspacher and W. Brestovansky, Generalized Schur numbers, in D. Jungnickel and K. Vedder (eds), \textit{Combinatorial Theory}, Lecture Notes in Mathematics, vol 969, Springer, Berlin, Heidelberg, 1982, 30-38.

\bibitem{BMRS2019} L. Boza, J. M. Mar\'{i}n, M. P. Revuelta, and M. I. Sanz, 3-color Schur numbers, \textit{Discrete Appl. Math.} \textbf{263} (2019), 59-68.

\bibitem{DMT2022} A. Dileep, J. Moondra, and Tripathi, New proofs for the disjunctive Rado number of the equations $x_1-x_2=a$ and $x_1-x_2=b$, \textit{Graphs Combin.} \textbf{38} (2022), 38.

\bibitem{DMT2024} A. Dileep, J. Moondra, and A. Tripathi, On disjunctive Rado numbers for some sets of equations, {\it Electron. J. Combin.} \textbf{31} (2024), no. 1, \#P1.69.

\bibitem{DT2025} S. Dwivedi and A. Tripathi, On the two-color disjunctive Rado number for the equations $\sum_{i=1}^{m-2}x_i+ax_{m-1}-x_m=c_j$, $j=1,2$, \textit{Integers} \textbf{25} (2025), \#A108.

\bibitem{EG2021} \"{O}. E\u{g}ecio\u{g}lu and A. Garsia,
\textit{Lessons in Enumerative Combinatorics}, Springer, Graduate Texts in Mathematics, 2021.

\bibitem{GGL2014} K. Gandhi, N. Golowich, and L. M. Lov\'{a}sz, Degree of regularity of linear homogeneous equations and inequalities, \textit{J. Comb.} \textbf{5} (2014), no. 2, 235-243.

\bibitem{HL1998} N. Hindman and I. Leader, Partition regular inequalities, \textit{European J. Combin.} \textbf{19} (1998), 573-758.

\bibitem{Irving1973} R. W. Irving, An extension of Schur’s theorem on sum-free partitions, \textit{Acta Arith.} \textbf{XXV} (1973), 55-64.

\bibitem{JS2005} B. Johnson and D. Schaal, Disjunctive Rado numbers, \textit{J. Combin. Theory Ser. A} \textbf{112} (2005), 263-276.

\bibitem{LR2014} B. M. Landman and A. Robertson, \textit{Ramsey Theory on the Integers}, Second Edition, Student Mathematical Library Volume 73, American Mathematical Society, Providence, 2014.

\bibitem{KS2003} W. Kosek and D. Schaal, A note on disjunctive Rado numbers, \textit{Adv. in Appl. Math.} \textbf{31} (2003), 433-439.

\bibitem{Rado1933} R. Rado, Studien zur Kombinatorik, \textit{Math. Z.} \textbf{36} (1933), no. 1, 424-480.

\bibitem{Robertson2000} A. Robertson, Difference Ramsey numbers and Issai numbers, \textit{Adv. in Appl. Math.} \textbf{25} (2000), 153-162.

\bibitem{Schur1916} I. Schur, \"{U}ber die Kongruenz $x^m+y^m\equiv z^m\pmod p$, \textit{Jahresber. Dtsch. Math.-Ver.} \textbf{25} (1916), 114-117.

\bibitem{Znam1966} S. Zn\'{a}m, Generalization of a number-theoretical result, \textit{Mat.-Fyz. Čas.} \textbf{16} (1966), no. 4, 357-361.

\bibitem{Znam1967} S. Zn\'{a}m, On $k$-thin sets and $n$-extensive graphs, \textit{Mat. Čas.} \textbf{17} (1967), no. 4, 297-307.
\end{thebibliography}
\end{document}